\documentclass{amsart}
\usepackage{amsfonts}

\newtheorem{theorem}{Theorem}
\theoremstyle{plain}

\newtheorem{corollary}{Corollary}

\newtheorem{definition}{Definition}

\newtheorem{lemma}{Lemma}

\numberwithin{equation}{section}
\input{tcilatex}

\begin{document}
\title[Hermite-Hadamard Type Inequalities]{HERMITE-HADAMARD TYPE
INEQUALITIES FOR FUNCTIONS WHOSE DERIVATIVES ARE STRONGLY $\varphi -$CONVEX}
\author{\.{I}mdat \.{I}\c{s}can}
\address{Department of Mathematics, Faculty of Arts and Sciences,\\
Giresun University, 28100, Giresun, Turkey.}
\email{imdat.iscan@giresun.edu.tr; imdati@yahoo.com}
\author{Erdal Unluyol}
\address{Department of Mathematics, Faculty of Arts and Sciences,\\
Ordu University, 52200, Ordu, Turkey.}
\email{erdalunluyol@odu.edu.tr; eunluyol@yahoo.com}
\date{May 26, 2012}

\begin{abstract}
In this paper several inequalities of the right-hand side of
Hermite-Hadamard's inequality are obtained for the class of functions whose
derivatives in absolutely value at certain powers are strongly $\varphi $%
-convex with modulus c\TEXTsymbol{>}0.
\end{abstract}

\maketitle

\section{Introduction}

\bigskip Let $f:I\subset \mathbb{R\rightarrow R}$ be a convex function
defined on the interval $I$ of real numbers and $a,b\in I$ with $a<b$, then

\begin{equation}
f\left( \frac{a+b}{2}\right) \leq \frac{1}{b-a}\dint\limits_{a}^{b}f(x)dx%
\leq \frac{f(a)+f(b)}{2}\text{.}  \label{1-1}
\end{equation}

This doubly inequality is known in the literature as Hermite-Hadamard
integral inequality for convex functions. This inequality plays an important
role in convex analysis and it has a huge literature dealing with its
applications, various generalizations and refinements (see \cite%
{BP03,BP04,DA98,DP00,KBOP07,PP00,SOD10} ).

Let us consider a function $\varphi :[a,b]\longrightarrow \lbrack a,b]$,
where $[a,b]\subset 
\mathbb{R}
.$Youness have defined the $\varphi -$convex functions in \cite{Y99} :

\begin{definition}
A function $f:[a,b]\mathbb{\rightarrow R}$ is said to be $\varphi -$convex
on $[a,b]$ if for every two points $x,y\in \lbrack a,b]$ and $t\in \left[ 0,1%
\right] $ the following inequality holds:%
\begin{equation*}
f\left( t\varphi (x)+(1-t)\varphi (y)\right) \leq tf\left( \varphi
(x)\right) +(1-t)f\left( \varphi (y)\right) .
\end{equation*}
\end{definition}

Obviously, if function $\varphi $ is the identity, then the classical
convexity is obtained from the previous definition.

Recall also that a function $f:I\mathbb{\rightarrow R}$ is called strongly
convex with modulus $c>0$, if%
\begin{equation*}
f\left( tx+(1-t)y\right) \leq tf\left( x\right) +(1-t)f\left( y\right)
-ct(1-t)\left( x-y\right) ^{2}
\end{equation*}%
for all $x,y\in I$ and $t\in \left( 0,1\right) $. Strongly convex functions
have been introduced by Polyak in \cite{P66} and they play an important role
in optimization theory and mathematical economics. Various properties and
applications of them can be found in the literature see \cite{MN10,NP11,P66}.

Let $\left( X,\left\Vert .\right\Vert \right) $ be a real normed space, $D$
stands for a convex subset of $X$, $\varphi :D\longrightarrow D$ is a given
function and $c$ is a positive constant. In \cite{S12} Sarikaya have
introduced the notion of the strongly $\varphi -$convex functions with
modulus $c$ and some properties of them. Moreover in his paper, Sarikaya
have presented a version Hermite-Hadamard-type inequalities for strongly $%
\varphi -$convex functions as follows:

\begin{definition}
A function $f:D\mathbb{\rightarrow R}$ is said to be strongly $\varphi -$%
convex with modulus $c$ if 
\begin{equation*}
f\left( t\varphi (x)+(1-t)\varphi (y)\right) \leq tf\left( \varphi
(x)\right) +(1-t)f\left( \varphi (y)\right) -ct(1-t)\left\Vert \varphi
(x)-\varphi (y)\right\Vert ^{2}
\end{equation*}%
for all $x,y\in D$ and $t\in \left[ 0,1\right] $.
\end{definition}

The notion of $\varphi -$convex function corresponds to the case $c=0.$If
function $\varphi $ is the identity, then the strongly convexity with
modulus $c>0$ is obtained from the previous definition

\begin{theorem}
\label{A.1}If $f:[a,b]\mathbb{\rightarrow R}$ is strongly $\varphi -$convex
with modulus $c>0$ for the continuous function $\varphi
:[a,b]\longrightarrow \lbrack a,b],$ then%
\begin{eqnarray*}
f\left( \frac{\varphi (a)+\varphi (b)}{2}\right) +\frac{c}{12}\left( \varphi
(b)-\varphi (a)\right) ^{2} &\leq &\frac{1}{\varphi (b)-\varphi (a)}%
\dint\limits_{\varphi (a)}^{\varphi (b)}f(x)dx \\
&\leq &\frac{f\left( \varphi (a)\right) +f\left( \varphi (b)\right) }{2}-%
\frac{c}{6}\left( \varphi (b)-\varphi (a)\right) ^{2}
\end{eqnarray*}
\end{theorem}

The main aim of this paper is to establish new inequalities of
Hermite-Hadamard type for the class of functions whose derivatives at
certain powers are strongly $\varphi -$convex with modulus c.

\section{Inequalities for functions whose derivatives are strongly $\protect%
\varphi -$convex with modulus c}

In order to prove our main resuls we need the following lemma:

\begin{lemma}
\label{2.1}Let $f:I\subset \mathbb{R\rightarrow R}$ be a differentiable
mapping on $I^{\circ }$, $a,b\in I$ with $a<b$. If $f^{\prime }\in L[a,b]$
and $\varphi :[a,b]\longrightarrow \lbrack a,b]$ a continuous function with $%
\varphi (a)<\varphi (b).$ then the following equality holds:\bigskip 
\begin{equation}
\frac{f\left( \varphi (a)\right) +f\left( \varphi (b)\right) }{2}-\frac{1}{%
\varphi (b)-\varphi (a)}\dint\limits_{\varphi (a)}^{\varphi (b)}f(x)dx=\frac{%
\varphi (b)-\varphi (a)}{2}\dint\limits_{0}^{1}\left( 2t-1\right) f^{\prime
}\left( t\varphi (b)+(1-t)\varphi (a)\right) dt  \label{1}
\end{equation}
\end{lemma}

\begin{proof}
By using partial integration in right hand of (\ref{1}) equality, the proof
is obvious.
\end{proof}

\bigskip The next theorem gives a new the upper Hermite-Hadamard inequality
for strongly $\varphi -$convex functions with modulus c as follows:

\begin{theorem}
\label{1.1}Let $f:I\subset \mathbb{R\rightarrow R}$ be a differentiable
mapping on $I^{\circ }$, $a,b\in I$ with $a<b$ such that $f^{\prime }\in
L[a,b].$ If $\left\vert f^{\prime }\right\vert ^{q}$ is strongly $\varphi -$%
convex functions with modulus c on $[a,b]$ for the continuous function $%
\varphi :[a,b]\longrightarrow \lbrack a,b]$ with $\varphi (a)<\varphi (b)$
and $q\geq 1$ then 
\begin{eqnarray}
&&\left\vert \frac{f\left( \varphi (a)\right) +f\left( \varphi (b)\right) }{2%
}-\frac{1}{\varphi (b)-\varphi (a)}\dint\limits_{\varphi (a)}^{\varphi
(b)}f(x)dx\right\vert  \label{2-2} \\
&\leq &\frac{\varphi (b)-\varphi (a)}{4}\allowbreak \left( \frac{\left\vert
f^{\prime }(\varphi (b))\right\vert ^{q}+\left\vert f^{\prime }(\varphi
(a))\right\vert ^{q}}{2}-\frac{c}{8}\left( \varphi (b)-\varphi (a)\right)
^{2}\right) ^{\frac{1}{q}}  \notag
\end{eqnarray}

\begin{proof}
Suppose that $q=1$. From Lemma \ref{2.1} and using the strongly $\varphi -$%
convexity functions with modulus c of $\left\vert f^{\prime }\right\vert $,
we have

\begin{eqnarray*}
&&\left\vert \frac{f\left( \varphi (a)\right) +f\left( \varphi (b)\right) }{2%
}-\frac{1}{\varphi (b)-\varphi (a)}\dint\limits_{\varphi (a)}^{\varphi
(b)}f(x)dx\right\vert \\
&\leq &\frac{\varphi (b)-\varphi (a)}{2}\dint\limits_{0}^{1}\left\vert
2t-1\right\vert \left\vert f^{\prime }(t\varphi (b)+(1-t)\varphi
(a))\right\vert dt \\
&\leq &\frac{\varphi (b)-\varphi (a)}{2}\dint\limits_{0}^{1}\left\vert
2t-1\right\vert \left[ t\left\vert f^{\prime }(\varphi (b))\right\vert
+(1-t)\left\vert f^{\prime }(\varphi (a))\right\vert -ct\left( 1-t\right)
\left( \varphi (b)-\varphi (a)\right) ^{2}\right] dt
\end{eqnarray*}

We have 
\begin{equation*}
\dint\limits_{0}^{1}\left\vert 2t-1\right\vert tdt=\frac{1}{4}%
,~~\dint\limits_{0}^{1}\left\vert 2t-1\right\vert (1-t)dt=\frac{1}{4}
\end{equation*}

and 
\begin{equation*}
\dint\limits_{0}^{1}\left\vert 2t-1\right\vert t\left( 1-t\right) dt=\frac{1%
}{16}
\end{equation*}

hence we obtain

\begin{eqnarray*}
&&\left\vert \frac{f\left( \varphi (a)\right) +f\left( \varphi (b)\right) }{2%
}-\frac{1}{\varphi (b)-\varphi (a)}\dint\limits_{\varphi (a)}^{\varphi
(b)}f(x)dx\right\vert \\
&\leq &\frac{\varphi (b)-\varphi (a)}{4}\left( \frac{\left\vert f^{\prime
}(\varphi (b))\right\vert +\left\vert f^{\prime }(\varphi (a))\right\vert }{2%
}-\frac{c}{8}\left( \varphi (b)-\varphi (a)\right) ^{2}\right)
\end{eqnarray*}

which copmletes the proof for this case.

Suppose now that $q\in \left( 1,\infty \right) $. From Lemma \ref{2.1} and
using the H\"{o}lder's integral inequality, we have%
\begin{eqnarray}
&&\dint\limits_{0}^{1}\left\vert 2t-1\right\vert \left\vert f^{\prime
}(t\varphi (b)+(1-t)\varphi (a))\right\vert dt  \label{2} \\
&\leq &\frac{\varphi (b)-\varphi (a)}{2}\left(
\dint\limits_{0}^{1}\left\vert 2t-1\right\vert dt\right) ^{\frac{q-1}{q}%
}\left( \dint\limits_{0}^{1}\left\vert 2t-1\right\vert \left\vert f^{\prime
}(t\varphi (b)+(1-t)\varphi (a))\right\vert ^{q}dt\right) ^{\frac{1}{q}} 
\notag
\end{eqnarray}

Since $\left\vert f^{\prime }\right\vert ^{q}$ is strongly $\varphi -$convex
functions with modulus c on $[a,b]$, we know that for every $t\in \left[ 0,1%
\right] $%
\begin{equation}
\left\vert f^{\prime }(t\varphi (b)+(1-t)\varphi (a))\right\vert ^{q}\leq
t\left\vert f^{\prime }(\varphi (b))\right\vert ^{q}+(1-t)\left\vert
f^{\prime }(\varphi (a))\right\vert ^{q}-ct\left( 1-t\right) \left( \varphi
(b)-\varphi (a)\right) ^{2}.  \label{3}
\end{equation}

From \ref{1}, \ref{2} and \ref{3}, we have

\begin{eqnarray*}
&&\left\vert \frac{f\left( \varphi (a)\right) +f\left( \varphi (b)\right) }{2%
}-\frac{1}{\varphi (b)-\varphi (a)}\dint\limits_{\varphi (a)}^{\varphi
(b)}f(x)dx\right\vert \\
&\leq &\frac{\varphi (b)-\varphi (a)}{4}\allowbreak \left( \frac{\left\vert
f^{\prime }(\varphi (b))\right\vert ^{q}+\left\vert f^{\prime }(\varphi
(a))\right\vert ^{q}}{2}-\frac{c}{8}\left( \varphi (b)-\varphi (a)\right)
^{2}\right) ^{\frac{1}{q}}
\end{eqnarray*}

which completes the proof.
\end{proof}
\end{theorem}

\begin{corollary}
Suppose that all the assumptions of Theorem\ref{1.1} are satisfied, in this
case:

\begin{enumerate}
\item For $c=0$, i.e. if $\left\vert f^{\prime }\right\vert ^{q}$ is $%
\varphi -$convex functions, we have%
\begin{eqnarray}
&&\left\vert \frac{f\left( \varphi (a)\right) +f\left( \varphi (b)\right) }{2%
}-\frac{1}{\varphi (b)-\varphi (a)}\dint\limits_{\varphi (a)}^{\varphi
(b)}f(x)dx\right\vert   \label{4} \\
&\leq &\frac{\varphi (b)-\varphi (a)}{4}\left( \frac{\left\vert f^{\prime
}(\varphi (b))\right\vert ^{q}+\left\vert f^{\prime }(\varphi
(a))\right\vert ^{q}}{2}\right) ^{\frac{1}{q}}  \notag
\end{eqnarray}

\item For $\varphi (t)=t$ in (\ref{4}) inequality :%
\begin{eqnarray*}
&&\left\vert \frac{f\left( a\right) +f\left( b\right) }{2}-\frac{1}{b-a}%
\dint\limits_{a}^{b}f(x)dx\right\vert \\
&\leq &\frac{b-a}{4}\left( \frac{\left\vert f^{\prime }(b)\right\vert
^{q}+\left\vert f^{\prime }(a)\right\vert ^{q}}{2}\right) ^{\frac{1}{q}}
\end{eqnarray*}%
we get the same result in \cite[Theorem 1]{PP00}.

\item If we take $\varphi (t)=t$ in (\ref{2-2}) for strongly convex
functions with modulus $c>0$ on $[a,b]$ inequality, then we obtain%
\begin{eqnarray*}
&&\left\vert \frac{f\left( a\right) +f\left( b\right) }{2}-\frac{1}{b-a}%
\dint\limits_{a}^{b}f(x)dx\right\vert \\
&\leq &\frac{b-a}{4}\allowbreak \left( \frac{\left\vert f^{\prime
}(b)\right\vert ^{q}+\left\vert f^{\prime }(a)\right\vert ^{q}}{2}-\frac{c}{8%
}\left( b-a\right) ^{2}\right) ^{\frac{1}{q}}
\end{eqnarray*}
\end{enumerate}
\end{corollary}

\begin{theorem}
\label{2.1.1}Let $f:I\subset \mathbb{R\rightarrow R}$ be a differentiable
mapping on $I^{\circ }$, $a,b\in I$ with $a<b$ such that $f^{\prime }\in
L[a,b].$ If $\left\vert f^{\prime }\right\vert ^{q}$ is strongly $\varphi -$%
convex functions with modulus $c$ on $[a,b]$ for the continuous function $%
\varphi :[a,b]\longrightarrow \lbrack a,b]$ with $\varphi (a)<\frac{\varphi
(a)+\varphi (b)}{2}<\varphi (b)$ and $q>1$ then 
\begin{eqnarray}
&&\left\vert \frac{f\left( \varphi (a)\right) +f\left( \varphi (b)\right) }{2%
}-\frac{1}{\varphi (b)-\varphi (a)}\dint\limits_{\varphi (a)}^{\varphi
(b)}f(x)dx\right\vert \leq \frac{\varphi (b)-\varphi (a)}{4}\left( \frac{1}{%
(p+1)}\right) ^{\frac{1}{p}}  \label{5} \\
&&\times \left( \frac{1}{2}\right) ^{\frac{1}{q}}\left[ 
\begin{array}{c}
\left( \left\vert f^{\prime }\left( \frac{\varphi (a)+\varphi (b)}{2}\right)
\right\vert ^{q}+\left\vert f^{\prime }\left( \varphi (a)\right) \right\vert
^{q}-\frac{c}{3}\left( \varphi (b)-\varphi (a)\right) ^{2}\right) ^{\frac{1}{%
q}} \\ 
+\left( \left\vert f^{\prime }\left( \frac{\varphi (a)+\varphi (b)}{2}%
\right) \right\vert ^{q}+\left\vert f^{\prime }\left( \varphi (b)\right)
\right\vert ^{q}-\frac{c}{3}\left( \varphi (b)-\varphi (a)\right)
^{2}\right) ^{\frac{1}{q}}%
\end{array}%
\right] ,  \notag
\end{eqnarray}%
where $\frac{1}{p}+\frac{1}{q}=1.$
\end{theorem}

\begin{proof}
From Lemma \ref{2.1} and using the H\"{o}lder inequality, we have%
\begin{eqnarray*}
&&\left\vert \frac{f\left( \varphi (a)\right) +f\left( \varphi (b)\right) }{2%
}-\frac{1}{\varphi (b)-\varphi (a)}\dint\limits_{\varphi (a)}^{\varphi
(b)}f(x)dx\right\vert \\
&\leq &\frac{\varphi (b)-\varphi (a)}{2}\left( \dint\limits_{0}^{\frac{1}{2}%
}\left( 1-2t\right) ^{p}dt\right) ^{\frac{1}{p}}\left( \dint\limits_{0}^{%
\frac{1}{2}}~\left\vert f^{\prime }(t\varphi (b)+(1-t)\varphi
(a))\right\vert ^{q}dt\right) ^{\frac{1}{q}} \\
&&+\frac{\varphi (b)-\varphi (a)}{2}\left( \dint\limits_{\frac{1}{2}%
}^{1}\left( 2t-1\right) ^{p}dt\right) ^{\frac{1}{p}}\left( \dint\limits_{%
\frac{1}{2}}^{1}\left\vert f^{\prime }(t\varphi (b)+(1-t)\varphi
(a))\right\vert ^{q}dt\right) ^{\frac{1}{q}} \\
&\leq &\frac{\varphi (b)-\varphi (a)}{2}\left( \frac{1}{2(p+1)}\right) ^{%
\frac{1}{p}}\left( \frac{1}{4}\right) ^{\frac{1}{q}}\left[ 
\begin{array}{c}
\left( \left\vert f^{\prime }\left( \frac{\varphi (a)+\varphi (b)}{2}\right)
\right\vert ^{q}+\left\vert f^{\prime }\left( \varphi (a)\right) \right\vert
^{q}-\frac{c}{3}\left( \varphi (b)-\varphi (a)\right) ^{2}\right) ^{\frac{1}{%
q}} \\ 
+\left( \left\vert f^{\prime }\left( \frac{\varphi (a)+\varphi (b)}{2}%
\right) \right\vert ^{q}+\left\vert f^{\prime }\left( \varphi (b)\right)
\right\vert ^{q}-\frac{c}{3}\left( \varphi (b)-\varphi (a)\right)
^{2}\right) ^{\frac{1}{q}}%
\end{array}%
\right]
\end{eqnarray*}%
where we use the fact that%
\begin{equation*}
\dint\limits_{0}^{\frac{1}{2}}\left( 1-2t\right) ^{p}dt=\dint\limits_{\frac{1%
}{2}}^{1}\left( 2t-1\right) ^{p}dt=\frac{1}{2(p+1)}
\end{equation*}%
and by Theorem\ref{A.1} we get%
\begin{eqnarray*}
\dint\limits_{0}^{\frac{1}{2}}~\left\vert f^{\prime }(t\varphi
(b)+(1-t)\varphi (a))\right\vert ^{q}dt &=&\frac{1}{2}\frac{2}{\varphi
(b)-\varphi (a)}\dint\limits_{\varphi (a)}^{\frac{\varphi (a)+\varphi (b)}{2}%
}\left\vert f^{\prime }(x)\right\vert ^{q}~dx \\
&\leq &\left( \frac{1}{4}\right) \left( \left\vert f^{\prime }\left( \frac{%
\varphi (a)+\varphi (b)}{2}\right) \right\vert ^{q}+\left\vert f^{\prime
}\left( \varphi (a)\right) \right\vert ^{q}-\frac{c}{3}\left( \varphi
(b)-\varphi (a)\right) ^{2}\right) , \\
\frac{\lambda +\mu }{\mu }\dint\limits_{\frac{\lambda }{\lambda +\mu }%
}^{1}\left\vert f^{\prime }(tb+(1-t)a)\right\vert ^{q}~dt &=&\frac{1}{2}%
\frac{2}{\varphi (b)-\varphi (a)}\dint\limits_{\frac{\varphi (a)+\varphi (b)%
}{2}}^{\varphi (b)}\left\vert f^{\prime }(x)\right\vert ^{q}~dx \\
&\leq &\left( \frac{1}{4}\right) \left( \left\vert f^{\prime }\left( \frac{%
\varphi (a)+\varphi (b)}{2}\right) \right\vert ^{q}+\left\vert f^{\prime
}\left( \varphi (b)\right) \right\vert ^{q}-\frac{c}{3}\left( \varphi
(b)-\varphi (a)\right) ^{2}\right) .
\end{eqnarray*}
\end{proof}

\begin{corollary}
Suppose that all the assumptions of Theorem\ref{2.1.1} are satisfied, in
this case:

\begin{enumerate}
\item Since $\left\vert f^{\prime }\right\vert ^{q}$ is strongly $\varphi -$%
convex functions with modulus $c,$ from (\ref{5}) inequality we get%
\begin{eqnarray*}
&&\left\vert \frac{f\left( \varphi (a)\right) +f\left( \varphi (b)\right) }{2%
}-\frac{1}{\varphi (b)-\varphi (a)}\dint\limits_{\varphi (a)}^{\varphi
(b)}f(x)dx\right\vert \\
&\leq &\frac{\varphi (b)-\varphi (a)}{4}\left( \frac{1}{(p+1)}\right) ^{%
\frac{1}{p}}\left( \frac{1}{2}\right) ^{\frac{1}{q}}\left[ 
\begin{array}{c}
\left( \frac{\left\vert f^{\prime }\left( \varphi (b)\right) \right\vert
^{q}+3\left\vert f^{\prime }\left( \varphi (a)\right) \right\vert ^{q}}{2}-%
\frac{7c}{12}\left( \varphi (b)-\varphi (a)\right) ^{2}\right) ^{\frac{1}{q}}
\\ 
+\left( \frac{3\left\vert f^{\prime }\left( \varphi (b)\right) \right\vert
^{q}+\left\vert f^{\prime }\left( \varphi (a)\right) \right\vert ^{q}}{2}-%
\frac{7c}{12}\left( \varphi (b)-\varphi (a)\right) ^{2}\right) ^{\frac{1}{q}}%
\end{array}%
\right]
\end{eqnarray*}

\item For $c=0$, i.e. if $\left\vert f^{\prime }\right\vert ^{q}$ is $%
\varphi -$convex functions, we have%
\begin{eqnarray}
&&\left\vert \frac{f\left( \varphi (a)\right) +f\left( \varphi (b)\right) }{2%
}-\frac{1}{\varphi (b)-\varphi (a)}\dint\limits_{\varphi (a)}^{\varphi
(b)}f(x)dx\right\vert   \label{6} \\
&\leq &\frac{\varphi (b)-\varphi (a)}{4}\left( \frac{1}{(p+1)}\right) ^{%
\frac{1}{p}}\left( \frac{1}{2}\right) ^{\frac{1}{q}}\left[ 
\begin{array}{c}
\left( \left\vert f^{\prime }\left( \frac{\varphi (a)+\varphi (b)}{2}\right)
\right\vert ^{q}+\left\vert f^{\prime }\left( \varphi (a)\right) \right\vert
^{q}\right) ^{\frac{1}{q}} \\ 
+\left( \left\vert f^{\prime }\left( \frac{\varphi (a)+\varphi (b)}{2}%
\right) \right\vert ^{q}+\left\vert f^{\prime }\left( \varphi (b)\right)
\right\vert ^{q}\right) ^{\frac{1}{q}}%
\end{array}%
\right]   \notag
\end{eqnarray}

\item For $\varphi (t)=t$ in (\ref{6}) inequality :%
\begin{eqnarray*}
&&\left\vert \frac{f\left( a\right) +f\left( b\right) }{2}-\frac{1}{b-a}%
\dint\limits_{a}^{b}f(x)dx\right\vert \\
&\leq &\frac{b-a}{4}\left( \frac{1}{(p+1)}\right) ^{\frac{1}{p}}\left(
\allowbreak \frac{1}{2}\right) ^{\frac{1}{q}}\left[ \left( \left\vert
f^{\prime }\left( \frac{a+b}{2}\right) \right\vert ^{q}+\left\vert f^{\prime
}\left( a\right) \right\vert ^{q}\right) ^{\frac{1}{q}}+\left( \left\vert
f^{\prime }\left( \frac{a+b}{2}\right) \right\vert ^{q}+\left\vert f^{\prime
}\left( b\right) \right\vert ^{q}\right) ^{\frac{1}{q}}\right]
\end{eqnarray*}%
we get the same result in \cite{KBOP07} for $s=1.$

\item If we take $\varphi (t)=t$ in (\ref{5}) for strongly convex functions
with modulus $c>0$ on $[a,b]$ inequality, then we obtain%
\begin{eqnarray*}
&&\left\vert \frac{f\left( a\right) +f\left( b\right) }{2}-\frac{1}{b-a}%
\dint\limits_{a}^{b}f(x)dx\right\vert \\
&\leq &\frac{b-a}{4}\left( \frac{1}{(p+1)}\right) ^{\frac{1}{p}}\left(
\allowbreak \frac{1}{2}\right) ^{\frac{1}{q}}\left[ 
\begin{array}{c}
\left( \left\vert f^{\prime }\left( \frac{a+b}{2}\right) \right\vert
^{q}+\left\vert f^{\prime }\left( a\right) \right\vert ^{q}-\frac{c}{3}%
\left( b-a\right) ^{2}\right) ^{\frac{1}{q}} \\ 
+\left( \left\vert f^{\prime }\left( \frac{a+b}{2}\right) \right\vert
^{q}+\left\vert f^{\prime }\left( b\right) \right\vert ^{q}-\frac{c}{3}%
\left( b-a\right) ^{2}\right) ^{\frac{1}{q}}%
\end{array}%
\right]
\end{eqnarray*}
\end{enumerate}
\end{corollary}

\begin{theorem}
\label{2.2}Let $f:I\subset \mathbb{R\rightarrow R}$ be a differentiable
mapping on $I^{\circ }$, $a,b\in I$ with $a<b$ such that $f^{\prime }\in
L[a,b].$ If $\left\vert f^{\prime }\right\vert ^{q}$ strongly $\varphi -$%
convex functions with modulus c on $[a,b]$ for the continuous function $%
\varphi :[a,b]\longrightarrow \lbrack a,b]$ with $\varphi (a)<\varphi (b)$
and $q>1$ then 
\begin{eqnarray}
&&\left\vert \frac{f\left( \varphi (a)\right) +f\left( \varphi (b)\right) }{2%
}-\frac{1}{\varphi (b)-\varphi (a)}\dint\limits_{\varphi (a)}^{\varphi
(b)}f(x)dx\right\vert  \label{7} \\
&\leq &\frac{\varphi (b)-\varphi (a)}{2}\left( \frac{1}{p+1}\right) ^{\frac{1%
}{p}}\left( \frac{\left\vert f^{\prime }(\varphi (b))\right\vert
^{q}+\left\vert f^{\prime }(\varphi (a))\right\vert ^{q}}{2}-\frac{c}{6}%
\left( \varphi (b)-\varphi (a)\right) ^{2}\right) ^{\frac{1}{q}}.  \notag
\end{eqnarray}%
where $\frac{1}{p}+\frac{1}{q}=1.$
\end{theorem}

\begin{proof}
From Lemma \ref{2.1} and using the H\"{o}lder's integral inequality, we have
\ 
\begin{eqnarray*}
&&\left\vert \frac{f\left( \varphi (a)\right) +f\left( \varphi (b)\right) }{2%
}-\frac{1}{\varphi (b)-\varphi (a)}\dint\limits_{\varphi (a)}^{\varphi
(b)}f(x)dx\right\vert \\
&\leq &\frac{\varphi (b)-\varphi (a)}{2}\dint\limits_{0}^{1}\left\vert
2t-1\right\vert \left\vert f^{\prime }(t\varphi (b)+(1-t)\varphi
(a))\right\vert dt \\
&\leq &\frac{\varphi (b)-\varphi (a)}{2}\left(
\dint\limits_{0}^{1}\left\vert 2t-1\right\vert ^{p}dt\right) ^{\frac{1}{p}%
}\left( \dint\limits_{0}^{1}\left\vert f^{\prime }(t\varphi (b)+(1-t)\varphi
(a))\right\vert ^{q}dt\right) ^{\frac{1}{q}} \\
&\leq &\frac{\varphi (b)-\varphi (a)}{2}\left( \frac{1}{(p+1)}\right) ^{%
\frac{1}{p}}\left( \dint\limits_{0}^{1}t\left\vert f^{\prime }(\varphi
(b))\right\vert ^{q}+(1-t)\left\vert f^{\prime }(\varphi (a))\right\vert
^{q}-ct\left( 1-t\right) \left( \varphi (b)-\varphi (a)\right) ^{2}dt\right)
^{\frac{1}{q}} \\
&=&\frac{\varphi (b)-\varphi (a)}{2}\left( \frac{1}{(p+1)}\right) ^{\frac{1}{%
p}}\left( \frac{\left\vert f^{\prime }(\varphi (b))\right\vert
^{q}+\left\vert f^{\prime }(\varphi (a))\right\vert ^{q}}{2}-\frac{c}{6}%
\left( \varphi (b)-\varphi (a)\right) ^{2}\right) ^{\frac{1}{q}}.
\end{eqnarray*}
\end{proof}

\begin{corollary}
Suppose that all the assumptions of Theorem\ref{2.2} are satisfied, in this
case:
\end{corollary}

\begin{enumerate}
\item For $c=0$, i.e. if $\left\vert f^{\prime }\right\vert ^{q}$ is $%
\varphi -$convex functions, we have%
\begin{eqnarray}
&&\left\vert \frac{f\left( \varphi (a)\right) +f\left( \varphi (b)\right) }{2%
}-\frac{1}{\varphi (b)-\varphi (a)}\dint\limits_{\varphi (a)}^{\varphi
(b)}f(x)dx\right\vert   \label{8} \\
&\leq &\frac{\varphi (b)-\varphi (a)}{2}\left( \frac{1}{p+1}\right) ^{\frac{1%
}{p}}\left( \frac{\left\vert f^{\prime }(\varphi (b))\right\vert
^{q}+\left\vert f^{\prime }(\varphi (a))\right\vert ^{q}}{2}\right) ^{\frac{1%
}{q}}  \notag
\end{eqnarray}

\item We obtained the same result in \cite[Theorem 2.3]{DA98} for $\varphi
(t)=t$ in (\ref{8}) inequality :%
\begin{eqnarray*}
&&\left\vert \frac{f\left( a\right) +f\left( b\right) }{2}-\frac{1}{b-a}%
\dint\limits_{a}^{b}f(x)dx\right\vert \\
&\leq &\frac{b-a}{2}\left( \frac{1}{(p+1)}\right) ^{\frac{1}{p}}\left( \frac{%
\left\vert f^{\prime }\left( a\right) \right\vert ^{q}+\left\vert f^{\prime
}\left( b\right) \right\vert ^{q}}{2}\right) ^{\frac{1}{q}}
\end{eqnarray*}

\item If we take $\varphi (t)=t$ in (\ref{7}) for strongly convex functions
with modulus $c>0$ on $[a,b]$ inequality, then we obtain%
\begin{eqnarray*}
&&\left\vert \frac{f\left( a\right) +f\left( b\right) }{2}-\frac{1}{b-a}%
\dint\limits_{a}^{b}f(x)dx\right\vert \\
&\leq &\frac{b-a}{4}\left( \frac{1}{(p+1)}\right) ^{\frac{1}{p}}\left( \frac{%
\left\vert f^{\prime }\left( a\right) \right\vert ^{q}+\left\vert f^{\prime
}\left( b\right) \right\vert ^{q}}{2}-\frac{c}{6}\left( \varphi (b)-\varphi
(a)\right) ^{2}\right) ^{\frac{1}{q}}
\end{eqnarray*}
\end{enumerate}


\begin{thebibliography}{99}
\bibitem{BP03} M. Bessenyei and Zs. P\'{a}les, \textit{Hadamard-type
inequalities for generalized convex functions}, Math. Inequal. Appl. 6/3,
279-392, 2003.

\bibitem{BP04} M.K. Bakula, J. Pe\v{c}ari\'{c}, \textit{Note on some
Hadamard type inequalities}, Journal of Inequalities in Pure and Applied
Mathematics, vol. 5, article 74, 2004.

\bibitem{DA98} S.S.Dragomir,R.P.Agarwal,\textit{Two inequalities for
differentiable mappings and applications to special means of real numbers
and to trapezoidal formula,} Applied Mathematics Letters, 11 (5) (1998)
91--95.

\bibitem{DP00} S.S. Dragomir and C.E.M. Pearce, \textit{Selected Topics on
Hermite-Hadamard Inequalities and Applications,} RGMIA Monographs, Victoria
University, 2000.

\bibitem{KBOP07} U.S. Kirmaci, M.K. Bakula, M.E. \"{O}zdemir and J. Pe\v{c}%
ari\'{c}, \textit{Hadamard-type inequalities for s-convex functions,}Applied
Mathematics and Computation 193 (2007) 26--35.

\bibitem{MN10} N. Merentes and K. Nikodem, \textit{Remarks on strongly
convex functions, }Aequationes Math. 80 (2010), no. 1-2, 193-199.

\bibitem{NP11} K. Nikodem and Zs. Pales, \textit{Characterizations of inner
product spaces be strongly convex functions, }Banach J. Math. Anal. 5
(2011), no. 1, 83-87.

\bibitem{PP00} C.E.M. Pearce, J. Pe\v{c}ari\'{c}, \textit{Inequalities for
differentiable mappings with application to special means and quadrature
formul\ae , }Applied Mathematics Letters, 13 (200), 51-55.

\bibitem{P66} B.T. Polyak, \textit{Existence theorems and convergence of
minimizing sequences in extremum problems with restictions, }Soviet Math
Dokl. 7 (1966), 72-75.

\bibitem{S12} M.Z. Sarikaya, \textit{On Hermite Hadamard-type inequalities
for strongly }$\varphi -$\textit{convex functions}, Studia Universitatis
Babes-Bolyai Mathematica, in press.

\bibitem{SOD10} E.Set, M.E. \"{O}zdemir and S.S. Dragomir, \textit{On
Hadamard-Type inequalities involving several kinds of convexity}, Journal of
Inequalities and Applications, Article ID 286845, 12 pages, 2010.

\bibitem{Y99} E.A. Youness,\textit{\ E-convex sets, E-convex functions and
E-convex programming,}Journal of Optimization Theory and Applications, 102,
2 (1999), 439-450.
\end{thebibliography}
\end{document}